\documentclass[11pt]{amsart}

\usepackage{amsthm,amsmath,amssymb,amscd,amsfonts,latexsym}

\theoremstyle{plain}
\newtheorem{theorem}{Theorem}[section]

\newtheorem{proposition}[theorem]{Proposition}

\newtheorem{corollary}[theorem]{Corollary}
\newtheorem{def-thm}[theorem]{Definition-Theorem}
\newtheorem{lemma}[theorem]{Lemma}

\theoremstyle{definition}
\newtheorem{definition}[theorem]{Definition}

\newtheorem{remark}[theorem]{Remark}

\newcommand{\RR}{\mathbb{R}}

\newcommand{\CC}{\mathbb{C}}

\DeclareMathOperator{\Ric}{Ric}
\DeclareMathOperator{\di}{div}
\DeclareMathOperator{\kod}{kod}

\begin{document}

\title[On $K_M$ and negative holomorphic sectional curvature]{On the canonical line bundle and negative holomorphic sectional curvature}

\begin{abstract} 
We prove that a smooth complex projective threefold with a K\"ahler metric of negative holomorphic sectional curvature has ample canonical line bundle. In dimensions greater than three, we prove that, under equal assumptions, the nef dimension of the canonical line bundle is maximal. With certain additional assumptions, ampleness is again obtained. The methods used come from both complex differential geometry and complex algebraic geometry.
\end{abstract}

\author{Gordon Heier, Steven S.~Y.~Lu, Bun Wong}

\address{Department of Mathematics\\University of Houston\\4800 Calhoun Road, Houston, TX 77204\\USA}
\email{heier@math.uh.edu}

\address{D\'epartment de Math\'ematiques\\Universit\'e Du Q\'ebec \`a Montr\'eal\\C.P. 8888\\Succursale Centre-Ville\\Montr\'eal, Qc H3C 3P8\\Canada}
\email{lu.steven@uqam.ca}

\address{Department of Mathematics\\UC Riverside\\900 University Avenue\\ Riverside, CA 92521\\USA}
\email{wong@math.ucr.edu}


\maketitle

\section{Introduction} 
In complex differential geometry the connection between the holomorphic sectional curvature tensor and the Ricci curvature tensor remains unclear. Many compact K\"ahler manifolds with negative definite Ricci curvature tensor have no metric with negative holomorphic sectional curvature. The converse is an open problem. To be precise, it is unknown whether there always exists a metric with negative definite Ricci tensor on a compact complex K\"ahler manifold with negative holomorphic sectional curvature. \par
In complex dimension two, one can answer this question affirmatively by means of the classification theory of compact complex surfaces. The details of this folklore fact can be found e.g.\ in \cite[Theorem 3.1]{Wong}. Alternatively, in Section \ref{2d_case}, we will give a short direct proof for the two-dimensional case that just uses standard algebraic geometry and a theorem of Bishop-Goldberg. This proof also gives a good idea of the general nature of the issues involved. Generalizing it to higher dimension would be very desirable, but seems difficult due the use of the Riemann-Roch theorem and the lack of information on the Chern numbers.\par
The main theorem of this paper is the following affirmative answer for algebraic threefolds.
\begin{theorem}\label{mthm}
Let $M$ be a smooth projective threefold with a K\"ahler metric of negative holomorphic sectional curvature. Then its canonical line bundle $K_M$ is ample.
\end{theorem}
Due to the work of Yau \cite{Yau} (and the Kodaira embedding theorem), it is now well known that the canonical line bundle of a compact K\"ahler manifold $M$ is ample if and only if there exists (in every K\"ahler class on $M$) a K\"ahler metric of negative definite curvature.\par
Our method is a combination of differential geometric and algebraic geometric machinery. An important step is to eliminate the case when the Kodaira dimension is equal to zero. The paper by Peternell \cite{Peternell} contains a weaker result, but under the more relaxed assumption that $M$ is Brody hyperbolic.\par
In dimension greater than three, we can establish the conclusion of Theorem \ref{mthm} only with an additional (strong) assumption:
\begin{theorem}\label{core_alb_thm}
Let $M$ be a smooth projective manifold with a K\"ahler metric of negative holomorphic sectional curvature. Assume that the essential dimension or the Albanese dimension of $M$ is greater than $\dim M - 4$. Then its canonical line bundle $K_M$ is ample.
\end{theorem}
Without additional assumptions, we prove the following weaker positivity statement for $K_M$. 
\begin{theorem}\label{max_nef_dim_thm}
Let $M$ be a smooth projective manifold with a K\"ahler metric with negative holomorphic sectional curvature. Then the nef dimension of $M$ is equal to the dimension of $M$.
\end{theorem}

\section{Differential geometric background}\label{diff_geom_backgr}
Let $M$ be an $n$-dimensional complex manifold with local coordinates $z_1,\ldots,z_n$. Let
\begin{equation*}
ds^2=\sum_{i,j=1}^n g_{i\bar j} dz_i\otimes d\bar{z}_j
\end{equation*}
be a hermitian metric on $M$. The components $R_{i\bar j k \bar l}$ of the curvature tensor $R$ associated with the hermitian connection are locally given by the formula
\begin{equation*}
R_{i\bar j k \bar l}=-\frac{\partial^2 g_{i\bar j}}{\partial z_k\partial \bar z _l}+\sum_{p,q=1}^n g^{p\bar q}\frac{\partial g_{i\bar p}}{\partial z_k}\frac{\partial g_{q\bar j}}{\partial \bar z_l}.
\end{equation*}
The Ricci curvature tensor $\Ric$ is defined by
\begin{equation*}
\Ric_{k\bar l} = \sum_{i,j=1}^n g^{i\bar j} R_{i\bar j k\bar l},
\end{equation*}
and the scalar curvature function $s$ is defined by 
\begin{equation*}
s = \sum_{i,j=1}^n g^{k\bar l} \Ric_{k\bar l}.
\end{equation*}
If $\xi=\sum_{i=1}^n\xi_i \frac{\partial }{\partial z_i}$ is a complex unit vector at $p\in M$, then the holomorphic sectional curvature $H(\xi)$ in the direction of $\xi$ is given by
\begin{equation*}
H(\xi)=2 \sum_{i,j,k,l=1}^n R_{i\bar j k \bar l}(p)\xi_i\bar\xi_j\xi_k\bar \xi_l.
\end{equation*}
In particular,
\begin{equation*}
H\left(\frac{\partial }{\partial z_i}(p)\right)=2 R_{i\bar i i \bar i}(p).
\end{equation*} \par
An important fact about holomorphic sectional curvature is the following. If $M'$ is a submanifold of $M$, then the holomorphic sectional curvature of $M'$ does not exceed that of $M$. To be precise, if $\xi$ is a unit tangent vector to $M'$, then
\begin{equation*}
H'(\xi)\leq H(\xi),
\end{equation*}
where $H'$ is the holomorphic sectional curvature associated to the metric on $M'$ induced by $ds^2$. For a short proof of this inequality see \cite[Lemma 1]{Wu}. Basically, the inequality is an immediate consequence of the Gauss-Codazzi equation.\par
Compact hermitian manifolds with negative holomorphic sectional curvature enjoy the remarkable property that they do not accommodate an entire complex curve. A generalized Schwarz lemma due to Ahlfors implies the following useful result.
\begin{theorem}\label{hyp_theorem}
Let $M$ be a compact hermitian manifold with holomorphic sectional curvature bounded above by a negative constant. Then there exists no non-constant holomorphic map from the complex plane into $M$ (i.e., $M$ is Brody hyperbolic).
\end{theorem} 
In particular, on a compact hermitian manifold with negative holomorphic sectional curvature, there exist no (birational images of) rational or elliptic curves.\par
Although it is in general unknown if one can find a metric with negative definite Ricci curvature on a compact hermitian manifold with negative holomorphic sectional curvature, we have the following pointwise result due to Berger \cite{Berger}.
\begin{theorem}\label{berger_theorem}
The scalar curvature of a K\"ahler metric with negative holomorphic sectional curvature is also negative.
\end{theorem}
We recall that a hermitian metric $\sum_{i,j=1}^n g_{i\bar j} dz_i\otimes d\bar{z}_j$ is K\"ahler if and only if locally there exists a real-valued function $f$ such that $g_{i\bar j}=\frac{\partial^2 f}{\partial z_i\partial \bar z_j}$.\par
Berger's theorem is proven using a pointwise formula expressing the scalar curvature at a point in terms of the average holomorphic sectional curvature on the unit sphere in the tangent space at that point. It does not seem to be a strong enough statement to immediately yield results concerning the Ricci curvature. Nevertheless, it allows us to prove the following fact, which will play a crucial role in our argument to rule out the Calabi-Yau case. Since we were unable to find a proof in the literature, we provide one here.
\begin{theorem}\label{no_neg_hol_sect_curvature}
There exists no K\"ahler metric with negative holomorphic sectional curvature on a compact K\"ahler manifold $M$ with zero first real Chern class.
\end{theorem}
\begin{proof}
We will make use of the following formula, which is a special case of a formula due to Chern. Let
\begin{equation*}
u=\frac{v_2}{v_1},
\end{equation*}
where $v_1$, $v_2$ are the respective volume forms associated to two metrics $ds^2_1,ds^2_2$ on $M$. Then \cite[Formula (3)]{Chern} reads
\begin{equation*}
\frac 1 2 \Delta \ln u = s - Tr(\Ric),
\end{equation*}
where $\Delta$ is the Laplacian operator associated with $ds^2_1$, $Tr$ denotes the trace with respect to $ds^2_1$, $s$ is the scalar curvature of $ds^2_1$, and $\Ric$ is the Ricci curvature tensor of $ds^2_2$.\par
Assume $ds^2_1$ is a K\"ahler metric with negative holomorphic sectional curvature. By Theorem \ref{berger_theorem}, the scalar curvature of $ds^2_1$ is negative. Moreover, Yau's Theorem \cite{Yau} says that there exists a K\"ahler metric $ds^2_2$ with vanishing Ricci curvature tensor.\par
Let $p$ be the point at which $\ln u$ achieves its minimum. By the second derivative test, $(\Delta \ln u)(p)\geq 0$. On the other hand, since $\Ric\equiv 0$, it follows from Chern's formula that
\begin{equation*}
\frac 1 2 (\Delta \ln u )(p) = s(p) < 0.
\end{equation*}
Contradiction.
\end{proof}
\begin{remark}
Since we will see in Step 1 of the proof Theorem \ref{mthm} that a compact manifold $M$ with a K\"ahler metric of negative holomorphic sectional curvature has nef $K_M$,
we note that Theorem \ref{no_neg_hol_sect_curvature} can be rephrased as saying that a compact manifold $M$ with a K\"ahler metric of negative holomorphic sectional curvature has strictly positive numerical dimension $\nu(M)$. Recall that the {\it numerical dimension} $\nu(M)$ is $\max\{k\in\{0,1,\ldots,\dim M\}: (c_1^\RR(M))^k\not\equiv 0\}$ (when $K_M$ is nef).
\end{remark}
We conclude this section with the following corollary to Theorem \ref{no_neg_hol_sect_curvature}.
\begin{corollary}
Let $M$ be a smooth projective manifold of Picard number one with a K\"ahler metric of negative holomorphic sectional curvature. Then $K_M$ is ample.
\end{corollary}
\begin{proof}
We will see in Step 1 of the proof of Theorem \ref{mthm} that $K_M$ is nef. Since $K_M$ is not numerically trivial by Theorem \ref{no_neg_hol_sect_curvature}, its numerical class must be a positive multiple of the numerical class of an arbitrarily chosen ample divisor. According to the Nakai-Moishezon-Kleiman criterion \cite[Theorem 1.2.23]{Laz_I}, this shows that $K_M$ is ample.
\end{proof}

\section{Proof of Theorem \ref{mthm}} \label{proof_mthm}
The major task is to deal with the case when the Kodaira dimension of $M$, denoted by $\kod(M)$, is equal to zero. We recall the definition.
\begin{definition}
Let $M$ be an projective manifold and $K_M$ its canonical line bundle. If $H^0(M,mK_M)=0$ for all $m>0$, then $\kod(M)=-\infty$. If this is not the case, then we take $\kod(M)$ to be the maximal dimension of the images of $M$ under the rational maps furnished by non-empty linear systems $|mK_M|$ ($m>0$).
\end{definition}

We break our arguments into five steps as follows.\par
Step $1$. Recall that a line bundle $L$ on a projective manifold is {\it nef} if $L.C\geq 0$ for all curves $C$ on $M$. If the canonical line bundle $K_M$ of $M$ is not nef, then there exists a curve $C$ such that $K_M.C<0$. Then Mori's theory (\cite{Mori}, \cite{Kollar_Mori}) yields that there exists a rational curve on $M$ (regardless of $\dim M$). This is disallowed by Theorem \ref{hyp_theorem}. Hence we can assume that $K_M$ is nef from now on.\par
Step $2$. The following is a theorem of Miyaoka \cite[Theorem 1.1]{Miyaoka} that yields the existence of some pluricanonical sections.
\begin{theorem}
Let $M$ be a smooth projective threefold with nef canonical line bundle. Then $H^0(M,mK_M)$ is non-zero for some $m>0$. In other words, $\kod(M)\geq 0$.
\end{theorem}
Step 3. If $1\leq \kod(M)\leq 2$, we consider the Kodaira-Iitaka map $\pi:\tilde M\to Y$ (see \cite{Iitaka}). This is a surjective holomorphic map of smooth projective manifolds such that
\begin{enumerate}
\item $\tilde M$ is birational to $M$,
\item $\kod(Y)=\kod(M)$,
\item a general fiber $F$ of $\pi$ is a smooth irreducible manifold with $\kod(F)=0$.
\end{enumerate}
Firstly, if $F$ is a smooth curve with $\kod(F)=0$, then $F$ is an elliptic curve. Thus, $\tilde M$ and also its birational image $M$ are not Brody hyperbolic, which is a contradiction to Theorem \ref{hyp_theorem}.\par
Secondly, $F$ may be a smooth surface with $\kod(F)=0$. From the Enriques-Kodaira classification we know that $F$ can only be a hyperelliptic surface, a torus, a projective K3 surface or an Enriques surface. However, none of these surfaces are Brody hyperbolic. Hence $\tilde M$ and $M$ are not Brody hyperbolic, which is a contradiction. The details of this argument are as follows.\par
A hyperelliptic surface is a locally trivial fibration whose typical fiber is an elliptic curve. Thus, it is not Brody hyperbolic. \par
A torus is covered by the images of non-constant entire curves simply because its universal cover is $\CC^2$. Consequently, it cannot be Brody hyperbolic.\par
It is a theorem of Bogomolov-Mumford \cite[Appendix]{MM} that a projective K3 surface contains a pencil of singular elliptic curves. Consequently, it cannot be Brody hyperbolic.\par
Finally, the universal cover of any Enriques surface is a projective K3 surface, so $F$ cannot be Brody hyperbolic in this case either.\par
Step 4. If $\kod(M)=3$, then $K_M$ is big (and nef). The following proposition is a consequence of \cite{Kawamata_85} or \cite{Kawamata_91} or \cite[Exercise 10.1]{Debarre} and proves Theorem \ref{mthm} in this case.
\begin{proposition}\label{big_case}
Let $M$ be a complex projective Brody hyperbolic manifold with $K_M$ nef and big. Then $K_M$ is ample.
\end{proposition}
Step 5. The final  and most important case to rule out is $\kod(M)=0$. Once we have proven that all such algebraic threefolds cannot admit a K\"ahler metric with negative holomorphic sectional curvature, the proof of Theorem \ref{mthm} is complete. \par
Since $\kod(M)=0$, there exists $m>0$ with $\dim H^0(M,mK_M)=1$. Let $\sigma \in H^0(M,mK_M)\backslash\{0\}$. If $\di(\sigma) = \emptyset$, then $mK_M$ is the trivial line bundle. In particular, its first Chern class is zero. Consequently, the first real Chern class of $K_M$ is zero. This is impossible due to Theorem \ref{no_neg_hol_sect_curvature}.\par
We finally have to treat the situation when $\di(\sigma) = \sum_{i=1}^T k_iD_i$, where $T\geq 1$, $k_i \geq 1$ and $D_i$ a prime divisor. In order to do so, we need the following result due to Wilson \cite{wilson}.
\begin{theorem}
With the above notation, for a given index $i$, one of the following is true.
\begin{enumerate}
\item $D_i$ is a birational to a ruled surface.
\item $D_i$ has at worst rational double point singularities. Its desingularization $\hat D_i$ satisfies $\kod(\hat D_i)=0$.
\end{enumerate}
\end{theorem}
In our situation, the case (i) cannot occur, since it would constitute on obvious contradiction to the Brody hyperbolicity. As for case (ii), we saw above that $\hat D_i$ and $D_i$ cannot be Brody hyperbolic. Contradiction.
\section{Proofs of Theorems \ref{core_alb_thm} and \ref{max_nef_dim_thm}}
\subsection{The case of the essential dimension}
The essential dimension is based on Campana's construction of the {\it core} of a smooth projective manifold in \cite{campana_core} (see also \cite{Lu_long} for an alternative but equivalent construction). We think of the core as described in \cite[Theorem 3.3]{campana_core}, namely as an almost holomorphic fibration $c: M\dashrightarrow C(M)$, where $C(M)$ is an orbifold which is of general type in a suitable sense. Moreover, a general fiber $F$ of $c$ is {\it special}. We refer to \cite{campana_core} for the exact definition of special. It suffices for us to note that, in particular, special projective manifolds are not of maximal Kodaira dimension. The dimension of $C(M)$ is called the {\it essential dimension} of $M$.\par
The proof of Theorem \ref{core_alb_thm} for the case of the essential dimension is basically immediate from Theorem \ref{mthm}. We assume that the Kodaira dimension of $M$ is not maximal and derive a contradiction. If the essential dimension is at least $\dim M - 4$, then the dimension of $F$ is $1,2$, or $3$. By the properties of $c$, $F$ is not of general type. On the other hand, due to the curvature decreasing property described in Section \ref{diff_geom_backgr}, $F$ itself is a projective manifold with a K\"ahler metric of negative holomorphic sectional curvature. The case of $\dim F = 1$ being trivial, we can conclude the proof by noting that our results from Section \ref{2d_case} ($\dim F=2$) and Theorem \ref{mthm} ($\dim F = 3$) show that $K_F$ is ample. Contradiction.\par

\subsection{The case of the Albanese dimension}
Let $a:M\to T$ be the Albanese mapping of $M$ to its Albanese torus $T$ (see \cite{GH}). Then the {\it Albanese dimension} of $M$ is defined to be $\dim a(M)$. If the Albanese dimension of $M$ is greater than $\dim M -4$, then the map $a:M\to a(M)$ is a holomorphic map such that an irreducible component $F$ of a general fiber has dimension at most $3$. W.l.o.g., we assume that the dimension of $F$ is equal to $3$. By the same token as above, $\kod(F)=\dim F$. Again, we assume that the Kodaira dimension of $M$ is not maximal and derive a contradiction.\par
Let $\pi:\tilde M\to Z$ be a holomorphic version of the Kodaira-Iitaka map of $M$ and $\sigma: \tilde M\to M$ the pertaining modification of $M$. We denote by $\tilde F$ the strict transform of $F$ under $\sigma$, which also has maximal Kodaira dimension.\par
The restriction of $\pi$ to $\tilde F$ gives a holomorphic map $\pi: \tilde F\to\pi(\tilde F)$. If we denote an irreducible component of a general fiber of this map by $S$, then the Easy Addition Formula (see \cite{iitaka_gtm}) yields
\begin{equation*}
3=\kod ( \tilde F) \leq \kod (S)+\dim \pi( \tilde F).
\end{equation*}
Since clearly $3=\dim \tilde F=\dim S +\dim \pi( \tilde F)$, we see that $S$ is of maximal Kodaira dimension.\par
On the other hand, let $G$ be the fiber of $\pi:\tilde M\to Z$ such that $S$ is a component of $G\cap\tilde F$. By the standard properties of the Kodaira-Iitaka map, when $F$ and $S$ are appropriately chosen, $G$ is a smooth projective manifold with $\kod(G)=0$. By the subadditivity theorem of Koll\'ar \cite{kollar} and Viehweg \cite{viehweg} for the case of fibers of general type, we have
\begin{equation*}
0=\kod(G)\geq \kod (S)+\kod(a(\sigma(G))).
\end{equation*}
Note that $a(\sigma(G))$ may be singular, so we define its Kodaira dimension to  be that of a desingularization. By \cite[Theorem 10.9]{Ueno}, we know that $\kod(a(\sigma(G)))$ is at least $0$. Hence, since $S$ is of maximal Kodaira dimension, $0=\kod (S)=\kod(a\circ \sigma(G))$. Consequently, $S$ is a point.\par
The fact that $S$ is a point implies that $a\circ \sigma: G\to a(\sigma(G))$ is a generically finite holomorphic map. If we denote by $b$ the Albanese map of $G$, then Kawamata's \cite[Theorem 1]{kawamata_81} yields that $b$ is a fibration. By the universal property of $b$, the holomorphic map $a\circ \sigma: G\to a(\sigma(G))\subset T$ factors through $b$. However, this clearly means that the general fibers of $b$ are zero-dimensional. Consequently, $G$ is birational to its Albanese torus, whence $G$ and $\sigma(G)$ are not hyperbolic. Thus, $\tilde M$ and $M$ are not hyperbolic. Contradiction.

\subsection{The nef dimension: Proof of Theorem \ref{max_nef_dim_thm}}

Since it is very difficult to produce a section of $mK_M$ solely based on the fact that $K_M$ is nef, it is tempting to consider the nef dimension and nef reduction map introduced by Tsuji (\cite{Tsuji}, see also \cite{campana_81}, \cite{campana_94}, \cite{9aut}).\par
Let $M$ be a smooth projective manifold with nef canonical line bundle. We define two distinct points $p$ and $q$ of $M$ to be equivalent (``$p\sim q$'') if and only if there exists a connected (not necessarily irreducible) curve $C$ containing both $p$ and $q$ such that $K_M.C=0$. According to \cite{campana_94}, there is an almost holomorphic rational map $f:M\dashrightarrow Y$ with connected fibers to a normal projective variety $Y$ such that for generic points $p,q$ in $M$, $p\sim q$ if and only if $f(p)=f(q)$. The precise statement from \cite{9aut} on the nef reduction map is the following.
\begin{theorem}
Let $L$ be a nef line bundle on a normal projective variety $M$. Then there exists an almost holomorphic dominant rational map $f:M\dashrightarrow Y$ with connected fibers, called a ``reduction map'' such that
\begin{enumerate}
\item $L$ is numerically trivial on all compact fibers $F$ of $f$ with $\dim F = \dim M - \dim Y$
\item for every general point $x\in M$ and every irreducible curve $C$ passing through $x$ with $\dim f(C) > 0$, we have $L.C > 0$.
\end{enumerate}
The map $f$ is unique up to birational equivalence of $Y$.
\end{theorem}
When we apply the above theorem with $L=K_M$, we call $n(M):=\dim Y$ the {\it nef dimension} of $M$. The advantage of dealing with the nef dimension is that its definition does not appeal to the existence of any non-trivial pluricanonical sections. \par We conclude this section with the 
\begin{proof}[Proof of Theorem \ref{max_nef_dim_thm}]
Assume that $n(M)< \dim M$, and let $k:=\dim M-n(M) > 0$. Then the general fiber $F$ of the above reduction map $f$ is a smooth projective manifold of dimension $k$ such that $K_M.C=0$ for all curves $C$ in $F$. By the adjunction formula, $K_M|_F=K_F$, and consequently $K_F$ is numerically trivial. However, this means that the first real Chern class of $F$ is trivial (see \cite[Remark 1.1.20]{Laz_I}). Due to the curvature decreasing property and Theorem \ref{no_neg_hol_sect_curvature}, we get a contradiction. (Note that if $n(M)=0$, we work with $F=M$.)
\end{proof}
\section{Alternative approaches}
\subsection{Use of the abundance conjecture}
The abundance conjecture in algebraic geometry says that on a projective manifold $M$ with nef canonical line bundle $K_M$, there exists an integer $m>0$ such that $mK_M$ is base point free. If this conjecture is true, then our Theorem 1.1 would be true in arbitrary dimension. The proof would go as follows.\par
Consider the holomorphic map $\pi: M\to Y$ furnished by the base point free linear system $|mK_M|$. If $0 < \kod(M) < \dim M$, the general fiber of $\pi$ is a smooth projective manifold $F$ with $0<\dim F < \dim M$ and $\kod(F)=0$. Moreover, due to the curvature decreasing property of submanifolds discussed in Section \ref{diff_geom_backgr}, the induced K\"ahler metric on $F$ has negative holomorphic sectional curvature. This can be ruled out by induction.\par
In the case $\kod (M) = 0$, we have $\dim H^0(M,mK_M)\in\{0,1\}$ for all $m>0$. When we take $m$ to be such that $mK_M$ is base point free, then $\dim H^0(M,mK_M)=1$, and $mK_M$ is clearly trivial. Consequently, $K_M$ has trivial first real Chern class, and one again obtains a contradiction due to Theorem \ref{no_neg_hol_sect_curvature}.\par
Finally, in the case $\kod (M) = \dim M$, one shows that $K_M$ is ample as we did in Step 4 of the proof of Theorem \ref{mthm}.
\par
Unfortunately, the abundance conjecture is open in dimension at least four, although some special cases are known, such as the case of irregular fourfolds \cite{Fujino}. Moreover, using the abundance conjecture seems to be overkill in any case, since a more specialized argument making better use of the negative curvature assumption seems much more appropriate. In this spirit, we show in Section \ref{2d_case} how to use the negative curvature assumption directly to prove our main theorem in two dimensions without any sort of fibration argument.\par

\subsection{The two-dimensional case}\label{2d_case}
Let $S$ be a complex projective surface with a K\"ahler metric of negative holomorphic sectional curvature. In this subsection, we show that $K_S$ is ample using only standard algebraic geometry and a theorem from \cite{BG}.\par
We have already noted that $K_S$ is nef. However, more is true:
\begin{lemma}\label{strictly_nef}
Let $S$ be as above. Then $K_S$ is strictly nef, i.e., $K_S.C>0$ for all curves $C$ in $S$.
\end{lemma}
\begin{proof}
Assume that there is a curve $C$ with $K_S.C\leq 0$. Since $K_S$ is nef, this means $K_S.C= 0$. Because of the Brody hyperbolicity, the arithmetic genus $g(C)$ is at least $2$, and thus by the adjunction formula
\begin{equation*}
2\leq 2g(C) -2 =K_S.C+C^2=C^2.
\end{equation*}
The Hodge Index Theorem (\cite{BPV}) tells us that $C^2>0$ implies that for any class $[\omega]\in H^{1,1}_\RR(S)$ such that $\int_C \omega =0$, we have $\int_M \omega\wedge \omega\leq 0$. Moreover, $\int_M \omega\wedge \omega = 0$ if and only if $[\omega]=0$.
However, since $K_S$ is nef, we know that $K_S^2 = \int_M c_1^\RR(K_S)\wedge c_1^\RR(K_S) \geq 0$ (\cite[Theorem 1.4.9. (Kleiman's theorem)]{Laz_I}). According to the above with $[\omega]=c_1^\RR(K_S)$, this is only possible if $c_1^\RR(K_S)=0$. This is a contradiction due to Theorem \ref{no_neg_hol_sect_curvature}.
\end{proof}
According to the Nakai-Moishezon-Kleiman criterion \cite[Theorem 1.2.23]{Laz_I}, $K_S$ is ample if we can show that $K_S^2 >0$. We formulate this as a lemma.
\begin{lemma}
Let $S$ be as above. Then $K_S^2 > 0$. 
\end{lemma}
\begin{proof}
Since we already know $K_S^2$ to be non-negative, we just need to assume $K_S^2 = 0$ and derive a contradiction.\par
Since $K_S$ is strictly nef according to Lemma \ref{strictly_nef}, there exists a very ample divisor $H$ such that $K_S.H > 0$.\par
We now write down the Riemann-Roch theorem for $2K_S$ (see \cite[p.472]{GH}):
\begin{equation}\label{RR}
h^0(2K_S)-h^1(2K_S)+h^2(2K_S)=\frac 1 {12} (K_S^2+c_2(S))+\frac 1 2 ((2K_S)^2-2K_S^2).
\end{equation}
Note that by Serre duality, $h^2(2K_S)=h^0(-K_S)$. If this number is different from zero, then there is an effective divisor $D$ in the linear system $|-K_S|$ such that $D.H \geq 0$. This is a contradiction to $D.H=-K_S.H<0$. Consequently, equation \eqref{RR} can be simplified to
\begin{equation*}
h^0(2K_S) \geq \frac 1 {12} c_2(S).
\end{equation*}
Due to the Gauss-Bonnet theorem, $c_2(S)$ is just the topological Euler characteristic of $S$. However, \cite[Theorem 1.3]{BG} says that any compact K\"ahler surface of definite (positive or negative) holomorphic sectional curvature has positive Euler characteristic. Thus we can conclude that $h^0(2K_S) > 0$. \par
Finally, to derive the contradiction, let $C\in |2K_S|$ be an effective divisor. Since $K_S$ is strictly nef by Lemma \ref{strictly_nef}, we have
\begin{equation*}
0< C.K_S=2K_S.K_S=2 K_S^2=2\cdot 0=0.
\end{equation*}
\end{proof}

\newcommand{\etalchar}[1]{$^{#1}$}

\end{document}